\documentclass{amsart}
\usepackage[utf8]{inputenc}
\usepackage{amsmath}
\usepackage{amssymb}
\usepackage{amsthm}
\usepackage{enumitem}
\usepackage{todonotes}
\usepackage{hyperref}
\usepackage{url}
\usepackage{xcolor}
\usepackage{graphicx}
\usepackage{comment}

\usepackage{tikz}
\usetikzlibrary{cd}

\setlist[enumerate]{label=\rm{(\arabic*)}, ref=(\arabic*)}

\DeclareMathOperator{\St}{St}
\DeclareMathOperator{\Aut}{Aut}
\DeclareMathOperator{\End}{End}

\DeclareMathOperator{\Sym}{Sym}
\DeclareMathOperator{\TC}{\mathcal{TC}}

\newcommand{\N}{\mathbb{N}}
\newcommand{\Z}{\mathbb{Z}}

\newcommand*{\M}{\mathcal{M}}

\newcommand*{\tile}[6]{
\draw (0,{(0+#1)*#2}) -- (0,{(1+#1)*#2}) node [midway, left] {#3};
\draw (0,{(1+#1)*#2}) -- (#2,{(1+#1)*#2}) node [midway, above] {#5};
\draw (0,{(0+#1)*#2}) -- (#2,{(0+#1)*#2}) node [midway, below] {#4};
\draw (#2,{(0+#1)*#2}) -- (#2,{(1+#1)*#2}) node [midway, right] {#6};
}

\newtheorem{thm}{Theorem}[section]
\newtheorem{lemma}[thm]{Lemma}
\newtheorem{prop}[thm]{Proposition}
\newtheorem{cor}[thm]{Corollary}

\theoremstyle{definition}

\newtheorem{defn}[thm]{Definition}

\newtheorem{notation}[thm]{Notation}

\newtheorem{thmx}{Theorem}[section]

\title{Bireversible automata generating lamplighter groups}
\author{Dominik Francoeur}
\thanks{This project was supported by the Spanish Ministry of Science and Innovation Grant CEX2019-000904-S funded by MCIN/AEI/ 10.13039/501100011033.}

\begin{document}

\begin{abstract}
For every non-trivial finite abelian group $A$, we exhibit a bireversible automaton generating the lamplighter group $A \wr \Z$.
\end{abstract}

\maketitle

\section{Introduction}

A \emph{Mealy automaton} is a tuple $\M=(Q,L,\lambda, \rho)$ where $Q$ and $L$ are finite sets called respectively the \emph{set of states} and the \emph{alphabet}, and $\lambda\colon Q\times L \rightarrow L$ and $\rho\colon Q\times L \rightarrow Q$ are two maps called respectively the \emph{output map} and the \emph{transition map}. The automaton $\M$ is called \emph{invertible} if, for all $q\in Q$, the map $\lambda(q,\cdot)\colon L \rightarrow L$ is a bijection, and it is called \emph{reversible} if the map $\rho(\cdot, l)\colon Q \rightarrow Q$ is a bijection for all $l\in L$. Lastly, $\M$ is said to be \emph{bireversible} if it is invertible, reversible, and if the map $\delta \colon Q\times L \rightarrow L \times Q$ sending $(q,l)$ to $(\lambda(q,l),\rho(q,l))$ is a bijection.

Let us denote by $L[[t]]$ the set of formal power series with coefficients in $L$, i.e. the  set of elements of the form $s_0+s_1t + s_2t^2+s_3t^3+\dots$, where $s_i\in L$. Let us also denote by $\sigma\colon L[[t]] \rightarrow L[[t]]$ the \emph{shift map}, i.e. the map that sends $f(t)=s_0+s_1t + s_2t^2+s_3t^3+\dots$ to $\sigma(f(t)) = s_1+s_2t+s_3t^2 + \dots$. To make the notation more legible, we will often denote  $\sigma(f(t))$ by $f^{\sigma}(t)$.

Every $q\in Q$ recursively defines a map $\tau_q\colon L[[t]]\rightarrow L[[t]]$ by
\[\tau_q(f(t)) = \lambda(q,s_0)+ \tau_{\rho(q,s_0)}(f^{\sigma}(t))t\] 
where $f(t)=s_0+f^{\sigma}(t)t$. Notice that if $\M$ is invertible, then $\tau_q$ is an invertible map for every $q\in Q$. Thus, in this case, one can look at the group of permutations of $L[[t]]$ generated by all the states of $\M$.

\begin{defn}\label{defn:AutomatonGroup}
Let $\M=(Q,L,\lambda, \rho)$ be an invertible Mealy automaton. The \emph{group generated by $\M$} is the subgroup $G_{\M}\leq \Sym(L[[t]])$ generated by the set $\{\tau_q \mid q\in Q\}$.
\end{defn}

Groups generated by (invertible Mealy) automata can exhibit a wide range of properties, some of which are quite exotic, such as being finitely generated infinite torsion \cite{GuptaSidki83} or of intermediate growth \cite{Grigorchuk83}. The presence of such rare behaviour was one of the motivation behind the current interest in groups generated by automata.

If one restricts oneself to groups generated by bireversible automata, the range of  possible behaviour appears to reduce considerably. For instance, such a group must contain a free subsemigroup, and thus be of exponential growth, as soon as it contains an element of infinite order \cite{Klimann18, FrancoeurMitrofanov21} (note that it is currently unknown whether being infinite is equivalent to containing an element of infinite order for such a group). In return, however, the bireversibility of the automaton gives these groups additional structure that makes them very interesting in other respects. Indeed, they are related to $CAT(0)$ square complexes \cite{GlasnerMozes05} and to the commensurator of the free group in the automorphism group of a regular tree \cite{MacedonskaNekrashevychSushchansky00}, and their action on the Cantor space $L[[t]]$ is essentially free, which makes them well-suited to computations of the spectrum of random walks on their Cayley graph (see \cite{KambitesSilvaSteinberg06}).

Currently, however, very few examples of groups generated by bireversible automata are known. The most notable among these are possibly the non-abelian free groups of finite rank, a result due to Steinberg, Vorobets and Vorobets \cite{SteinbergVorobetsVorobets11}, following previous work by Vorobets and Vorobets \cite{VorobetsVorobets10}. More recently, Skipper and Steinberg have obtained another important and quite different family of examples, namely groups of the form $A\wr \Z$, where $A$ is a finite abelian group \cite{SkipperSteinberg20}. Such groups are sometimes called \emph{lamplighter groups}. However, Skipper and Steinberg's result does not cover all possible lamplighter groups over finite abelian groups. Indeed, they require some technical condition on the Sylow $2$-subgroup of the group $A$ (we refer the reader to \cite{SkipperSteinberg20} for the precise condition). In particular, their result does not cover the classical lamplighter group $(\Z/2\Z)\wr \Z$. In this paper, we close this gap by proving the following result.


\begin{thmx}\label{thm:MainThm}
For every non-trivial finite abelian group $A$, there exists a bireversible automaton generating the lamplighter group $A\wr \Z$.
\end{thmx}

In addition to producing all the lamplighter groups $A\wr \Z$, we note that our construction, in the cases already covered by the result of Skipper and Steinberg, appears to yield different automata and thus might still be of interest.

The paper is organised as follows. In Section \ref{sec:CayleyMachines}, we describe the automata that will be used to generate lamplighter groups, which we call \emph{twisted Cayley machines}. In Section \ref{sec:Proof}, we obtain these automata in terms of operations on formal power series with coefficient in a finite abelian group, viewed as a module over the ring of formal power series with coefficient in the endomorphisms of the group, and we use this point of view to prove Theorem \ref{thm:MainThm}.

\section{Twisted Cayley machines}\label{sec:CayleyMachines}

Cayley machines were first introduced by Krohn and Rhodes \cite{KrohnRhodes68}. Briefly, if $M$ is a monoid whose operation is denoted by $\cdot\colon M\times M \rightarrow M$, then the Cayley machine associated to $M$ is the Mealy automaton $\mathcal{C}(M) = (M,M,\cdot, \cdot)$. In other words, states and letters in this automaton are given by elements of $M$, and for a state $m_1\in M$ and a letter $m_2\in M$, both the output function and the transition function return $m_1\cdot m_2$.

Silva and Steinberg have shown \cite{SilvaSteinberg05} that if $A$ is a non-trivial abelian group, then the group generated by the Cayley machine $\mathcal{C}(A)$ is the lamplighter group $A\wr \Z$. However, the automaton $\mathcal{C}(A)$ is not bireversible. In order to obtain bireversible automata generating lamplighter groups, we introduce what we call here \emph{twisted Cayley machines}, which are obtained from Cayley machines by taking as the alphabet a cartesian product of the monoid and twisting the action on the second component by the first.

\begin{defn}\label{defn:TwistedCayleyMachine}
Let $M$ be a finite monoid and let $\cdot\colon M\times M \rightarrow M$ denote its operation. The \emph{twisted Cayley machine} associated to $M$ is $\TC(M) = (M, M\times M, \lambda_M, \rho_M)$, where
\begin{align*}
\lambda_M \colon M\times \left(M\times M\right) &\rightarrow M\times M \\
(a, (b, c)) &\mapsto (a\cdot b, a\cdot b\cdot c)
\end{align*}
and
\begin{align*}
\rho_M \colon M\times (M\times M) &\rightarrow M \\
(a, (b, c)) &\mapsto a\cdot c.
\end{align*}
\end{defn}

The advantage of these twisted Cayley machines is that if $M$ is a finite group, then $\TC(M)$ is a bireversible automaton.

\begin{prop}\label{prop:IsBirversible}
If $M$ is a finite group, the automaton $\TC(M)$ is bireversible.
\end{prop}
\begin{proof}
For every $a\in M$, the map $\lambda_M(a, \cdot) \colon M\times M \rightarrow M\times M$ is invertible, with inverse
\begin{align*}
\lambda_M(a,\cdot)^{-1} \colon M\times M &\rightarrow M\times M \\
(b, c) &\mapsto (a^{-1} b, b^{-1}c)
\end{align*}
as is readily checked. Thus, the automaton $\TC(M)$ is invertible. Similarly, for every $(b, c)\in M\times M$, the map $\rho_M(\cdot, (b, c))\colon M\rightarrow M$ is invertible, with inverse
\begin{align*}
\rho_M(\cdot, (b,c))^{-1} \colon M &\rightarrow M \\
a &\mapsto ac^{-1}.
\end{align*}
Thus, the automaton $\TC(M)$ is reversible.

Lastly, the map 
\begin{align*}
\delta_M\colon M \times (M\times M) &\rightarrow (M\times M) \times M \\
(a, (b, c)) &\mapsto (\lambda_M(a, (b, c)), \rho_M(a, (b, c))) = ((ab, abc), ac)
\end{align*}
is invertible, with inverse
\begin{align*}
\delta_M^{-1}\colon (M\times M) \times M &\rightarrow M \times (M\times M) \\
((b,c), a) &\mapsto (ac^{-1}b, (b^{-1}ca^{-1}b, b^{-1}c)).
\end{align*}
Thus, $\TC(M)$ is bireversible.
\end{proof}

As we will see, it turns out that the twisted Cayley automaton of a non-trivial finite abelian group $A$ also generates the lamplighter group $A\wr \Z$. To prove this, however, it will be convenient to reinterpret everything in terms of formal power series, which we do in the next section.

\section{Formal power series and lamplighter groups}\label{sec:Proof}

As in the work of Skipper and Steinberg \cite{SkipperSteinberg20}, we will make use of formal power series to understand our groups. Our main innovation here is that, instead of equipping finite abelian groups with a ring structure, we will use the fact that every abelian group comes equipped with a natural module structure over the ring of its endomorphisms, which gives us more flexibility.

\subsection{Definitions and notations}\label{subsec:DefnsAndNotations}

For the remainder of this section, let $A$ be a non-trivial finite abelian group and let $R = \End(A\times A)$ be the ring of endomorphisms of the group $A\times A$. Let $(A\times A)[[t]]$ be the abelian group of formal power series over $A\times A$ and let $R[[t]]$ be the ring of formal power series over $R$ with commuting variable. Then, $(A\times A)[[t]]$ has a natural left-$R[[t]]$-module structure.

For any $h(t)\in (A\times A)[[t]]$, we can define the permutation
\begin{align*}
\alpha_{h(t)}\colon (A\times A)[[t]] &\longrightarrow (A\times A)[[t]]\\
f(t) &\longmapsto f(t)+h(t)
\end{align*}
which is simply addition by $h(t)$ in the abelian group $(A\times A)[[t]]$. Notice that the map that sends $h(t)$ to $\alpha_{h(t)}$ is an embedding of the abelian group $(A\times A)[[t]]$ into the group of bijections of $(A\times A)[[t]]$. We will denote by $T\leq \Sym((A\times A)[[t]])$ the image of this embedding.

For every $g(t)=g_0+g_1t+g_2t^2+\dots \in R[[t]]$, we can also define an endomorphism
\begin{align*}
\mu_{g(t)}\colon (A\times A)[[t]] &\longrightarrow (A\times A)[[t]]\\
f(t) &\longmapsto g(t)f(t)
\end{align*}
of the group $(A\times A)[[t]]$, which is simply left multiplication by $g(t)$. Notice that $\mu_{g(t)}$ is an automorphism of $(A\times A)[[t]]$ if and only if $g(t)$ is invertible, which is the case if and only if $g_0$ is invertible. The map that sends $g(t)$ to $\mu_{g(t)}$ is an embedding of the group of units $R[[t]]^{\times}$ into the group $\Sym((A\times A)[[t]])$, and we will denote its image by $U$. As the next easy and well-known lemma shows, $T$ is normalised by $U$.

\begin{lemma}\label{lemma:UNormalisesT}
For every $g(t)\in R[[t]]^{\times}$ and every $h(t)\in (A\times A)[[t]]$, we have $\mu_{g(t)}\circ \alpha_{h(t)}\circ \mu_{g(t)}^{-1} = \alpha_{g(t)h(t)} \in T$. Consequently, using the notation described above, $T$ is normalised by $U$.
\end{lemma}

In what follows, we will make frequent use of elements of the form $\alpha_{h(t)}$ where $h(t) = (a,a)+(a,a)t+(a,a)t^2+\dots = \frac{1}{1-t}(a,a)$ for some $a\in A$. For convenience, we will introduce a shorthand notation for these elements.

\begin{notation}
For $a\in A$, we will denote by $\alpha_a$ the element $\alpha_{\frac{1}{1-t}(a,a)}\in \Sym((A\times A)[[t]])$. Since $a$ is an element of $A$ and not of $A\times A$, this notation is not ambiguous, and there should therefore be no confusion.
\end{notation}

\subsection{Twisted Cayley machines and lamplighter groups through power series}

Let $\xi\in R$ be the automorphism of $A\times A$ given by $\xi(a,b) = (a,a+b)$ and let $\zeta\in R$ be the endomorphism of $A\times A$ given by $\zeta(a,b)=(b,b)$. Let us define
\[\gamma(t) = \xi + \zeta t + \zeta t^2 + \zeta t^3+ \dots = \xi + \frac{\zeta t}{1-t} \in R[[t]]\]
Notice that $\xi$ is invertible in $R$, so $\mu_{\gamma(t)}$ is an element of $U$ (see Section \ref{subsec:DefnsAndNotations} above for the definition of $U$).

\begin{prop}\label{prop:AutomatonGroupPowerSeries}
For every $a\in A$ and for every $f(t)=(a_1,a_2) + f^{\sigma}(t)t$, we have
\begin{align*}
\alpha_{a}\circ \mu_{\gamma(t)}(f(t)) &= (a+a_1, a+ a_1+a_2) + \alpha_{a+a_2}\circ \mu_{\gamma(t)}(f^{\sigma}(t))t\\
&=\lambda_A(a,(a_1,a_2)) + \alpha_{\rho_A(a,(a_1,a_2))}\circ \mu_{\gamma(t)}(f^{\sigma}(t))t
\end{align*}
where the maps $\lambda_A$ and $\rho_A$ are those of Definition \ref{defn:TwistedCayleyMachine}, with the operation of $A$ denoted additively. Consequently, the subgroup of $\Sym((A\times A)[[t]])$ generated by the set $\{\alpha_{a}\circ \mu_{\gamma(t)} \mid a\in A\}$ is exactly the automaton group $G_{\TC(A)}$ of the twisted Cayley machine of $A$ (see Definition \ref{defn:AutomatonGroup}).
\end{prop}
\begin{proof}
We have
\begin{align*}
\alpha_a\circ \mu_{\gamma(t)}(f(t)) &= \gamma(t)f(t)+\frac{(a,a)}{1-t} \\
&= \gamma(t)(a_1,a_2) + \gamma(t)f^{\sigma}(t)t + \left((a,a) + \frac{(a,a)t}{1-t}\right) \\
&= \xi(a_1,a_2) + \frac{\zeta(a_1,a_2)t}{1-t} + \gamma(t)f^{\sigma}(t)t + \left((a,a) + \frac{(a,a)t}{1-t}\right)\\
&=(a+a_1, a+a_1+a_2) + \gamma(t)f^{\sigma}(t)t + \frac{t}{1-t}(a+a_2, a+a_2)\\
&=(a+a_1, a+ a_1+a_2) + \alpha_{a+a_2}\circ \mu_{\gamma(t)}(f^\sigma(t))t.
\end{align*}

Thus, by induction, we see that for every $a\in A$, the bijection $\tau_a\in \Sym((A\times A)[[t]])$ defined by the state $a$ of the twisted Cayley machine $\TC(A)$ is exactly the same bijection as $\alpha_a\circ \mu_{\gamma(t)}$.
\end{proof}

Thanks to Proposition \ref{prop:AutomatonGroupPowerSeries}, we now only have to prove that the subgroup of $\Sym((A\times A)[[t]])$ generated by $\{\alpha_a\circ \mu_{\gamma(t)} \mid a\in A\}$ is isomorphic to the lamplighter group $A\wr \Z$. We begin by observing that $A\wr \Z$ maps onto this group.

\begin{lemma}\label{lemma:HomLamplighterOntoG}
Let $\{s\}\cup A$ be the generating set for the lamplighter group $A\wr \Z = \left(\bigoplus_{\Z}A\right)\rtimes \Z$ where $s=1\in \Z$ and $A$ is identified with the copy of $A$ at position $0$ in $\left(\bigoplus_{\Z}A\right)$. Then, there exists a unique surjective homomorphism $\varphi\colon A\wr \Z \rightarrow G_{\TC(A)}$ such that $\varphi(s) = \mu_{\gamma(t)}$ and $\varphi(a)=\alpha_a$ for all $a\in A$.
\end{lemma}
\begin{proof}
The uniqueness of $\varphi$, if it exists, is obvious. By Proposition \ref{prop:AutomatonGroupPowerSeries}, the group $G_{\TC(A)}$ is generated by the set $\{\alpha_{a}\circ \mu_{\gamma(t)} \mid a\in A\}$. Using the fact that $\mu_{\gamma(t)} = \alpha_0\circ \mu_{\gamma(t)}$ is in this set, we conclude, multiplying every other generator by $\mu_{\gamma(t)}^{-1}$, that $\{\mu_{\gamma(t)}\}\cup \{\alpha_a\mid a\in A\}$ is also a generating set for $G_{\TC(A)}$. Thus, if the homomorphism $\varphi$ exists, it must be surjective.

It only remains to show that $\varphi$ exists. Let $A=\langle A \mid \mathcal{R} \rangle$ be a presentation of $A$. Then, we have the following presentation for the lamplighter group over $A$:
\[A\wr \Z = \langle s, A \mid \mathcal{R}, [s^{i}as^{-i}, b] \forall a,b\in A, \forall i\in \Z \rangle.\]
To show that $\varphi$ exists, it suffices to show that $G_{\TC(A)}$ satisfies the same relations. Since the map that sends $a$ to $\alpha_a$ is an isomorphism onto its image, it is clear that the relations $\mathcal{R}$ are satisfied. Furthermore, $\mu_{\gamma(t)}^{i}\circ \alpha_a \circ \mu_{\gamma(t)}^{-i}$ and $\alpha_b$ commute for every $a,b\in A$ and $i\in \Z$, since both are in $T$ by Lemma \ref{lemma:UNormalisesT}, which is an abelian group. This concludes the proof.
\end{proof}

To show that $G_{\TC(A)}$ is isomorphic to the lamplighter group $A\wr \Z$, it only remains to show that the homomorphism $\varphi$ of Lemma \ref{lemma:HomLamplighterOntoG} is injective. To do this, we will make use of the following easy facts about lamplighter groups.

\begin{lemma}\label{lemma:LamplighterAbelianNormalInSum}
Let $A\wr \Z = \left(\bigoplus_{\Z}A\right)\rtimes \Z$ be the lamplighter group over $A$, and let $N\trianglelefteq A\wr \Z$ be an abelian normal subgroup. Then, $N\trianglelefteq \bigoplus_{\Z}A$.
\end{lemma}
\begin{proof}
Let $\kappa s^{j}\in N$ be any element, where $\kappa\in \bigoplus_{\Z}A$ and $j\in \Z$. Let $\eta\in \bigoplus_{\Z}A$ be an element different from the identity (which exists, since $A$ is non-trivial). We have
\begin{align*}
[\eta\kappa s^{j}\eta^{-1}, \kappa s^{j}] &= \eta\kappa s^{j}\eta^{-1}\kappa s^{j} \eta s^{-j}\kappa^{-1} \eta^{-1} s^{-j} \kappa^{-1} \\
&=\eta \kappa (s^{j} \eta^{-1}\kappa s^{-j}) (s^{2j} \eta s^{-2j}) (s^{j}\kappa^{-1}\eta^{-1}s^{-j})\kappa^{-1}\\
&=\eta (s^{j} \eta^{-2} s^{-j}) (s^{2j} \eta s^{-2j}).
\end{align*}
Since $N$ is normal and abelian, we must have $[\eta\kappa s^{j}\eta^{-1}, \kappa s^{j}] =1$. However, from the computation above, this is only possible if $j=0$. This shows that every element of $N$ is in $\bigoplus_{\Z}A$.

\end{proof}

\begin{lemma}\label{lemma:SufficesToCheckKernelInSum}
Let $A\wr \Z = \left(\bigoplus_{\Z}A\right)\rtimes \Z$ be the lamplighter group over $A$, and let $N\trianglelefteq A\wr \Z$ be a normal subgroup. Then, $N$ is trivial if and only if its intersection with $\bigoplus_{\Z}A$ is trivial.
\end{lemma}
\begin{proof}
Let $N'\trianglelefteq N$ be the derived subgroup of $N$. Since the derived subgroup of $A\wr \Z$ is $\bigoplus_{\Z}A$, we must have $N'\leq \bigoplus_{\Z}A$. Thus, on the one hand, if $N'\ne 1$, the intersection between $N$ and $\bigoplus_{\Z}A$ is non-trivial. On the other hand, if $N'=1$, then $N$ is abelian, which means that it is contained in $\bigoplus_{\Z}A$ by Lemma \ref{lemma:LamplighterAbelianNormalInSum}.
\end{proof}

\begin{thm}
The automaton group $G_{\TC(A)}$ is isomorphic to $A\wr \Z$.
\end{thm}
\begin{proof}
To prove the result, we need to show that the homomorphism $\varphi\colon A\wr \Z \rightarrow G_{\TC(A)}$ of Lemma \ref{lemma:HomLamplighterOntoG} is injective. By Lemma \ref{lemma:SufficesToCheckKernelInSum}, it suffices to show that the restriction of $\varphi$ to $\bigoplus_{\Z} A$ is injective.

Let $k=(t^{i_1}a_1t^{-i_1})(t^{i_2}a_2t^{-i_2})\dots (t^{i_n}a_nt^{-i_n})\in \bigoplus_{\Z} A$ be an element in the kernel of $\varphi$, where $a_j\in A$ and $i_1< i_2 < \dots < i_n \in \Z$. Replacing $k$ by $t^{-i_1}kt^{i_i}$, we may assume without loss of generality that $i_1=0$. We then have
\[1=\varphi(k)=\alpha_{a_1} \circ (\mu_{\gamma(t)^{i_2}}\circ \alpha_{a_2}\circ \mu_{\gamma(t)^{-i_2}})\circ \dots \circ(\mu_{\gamma(t)^{i_n}}\circ \alpha_{a_n}\circ \mu_{\gamma(t)^{-i_n}}).\]
Thus, using Lemma \ref{lemma:UNormalisesT}, we find that 
\begin{align*}
0&=\frac{(a_1,a_1)}{1-t}+\gamma(t)^{i_2}\frac{(a_2,a_2)}{1-t} + \dots + \gamma(t)^{i_n}\frac{(a_n,a_n)}{1-t}\\
&=\frac{1}{1-t}(a_1,a_1) + \frac{(\xi+(\zeta-\xi)t)^{i_2}}{(1-t)^{i_2+1}}(a_2,a_2) + \dots + \frac{(\xi+(\zeta-\xi)t)^{i_n}}{(1-t)^{i_n+1}}(a_n,a_n)
\end{align*}
where we used the fact that $\gamma=\frac{\xi+(\zeta-\xi)t}{1-t}$ and that $t$ is a commuting variable. Multiplying this equation on the left by $(1-t)^{i_n+1}$, we obtain
\[0=(1-t)^{i_n}(a_1,a_1) + (1-t)^{i_n-i_2}(\xi + (\zeta-\xi)t)^{i_2}(a_2,a_2) + \dots + (\xi + (\zeta - \xi)t)^{i_n}(a_n,a_n).\]
Evaluating this expression in $t=1$, which we can do because $1$ commutes with every element of the ring $R[[t]]$, we find
\begin{align*}
0&=(\xi+(\zeta-\xi))^{i_n}(a_n,a_n) \\
&=\zeta^{i_n}(a_n,a_n)\\
&=(a_n,a_n)
\end{align*}
since $\zeta(a_n,a_n)=(a_n,a_n)$ from the definition of $\zeta$. This shows that $a_n=0$. By  induction, we conclude that $a_j=0$ for all $1\leq j \leq n$ and thus that $k=0$. Therefore, $\varphi$ is injective on $\bigoplus_{\Z} A$ and is consequently an isomorphism between $A\wr \Z$ and $G_{\TC(A)}$.
\end{proof}

\bibliographystyle{plain}
\bibliography{../biblio/biblio}

\begin{thebibliography}{10}

\bibitem{FrancoeurMitrofanov21}
Dominik Francoeur and Ivan Mitrofanov.
\newblock On the existence of free subsemigroups in reversible automata
  semigroups.
\newblock {\em Groups Geom. Dyn.}, 15(3):1103--1132, 2021.

\bibitem{GlasnerMozes05}
Yair Glasner and Shahar Mozes.
\newblock Automata and square complexes.
\newblock {\em Geom. Dedicata}, 111:43--64, 2005.

\bibitem{Grigorchuk83}
Rostislav~I. Grigorchuk.
\newblock On the {M}ilnor problem of group growth.
\newblock {\em Dokl. Akad. Nauk SSSR}, 271(1):30--33, 1983.

\bibitem{GuptaSidki83}
Narain Gupta and Sa\"{\i}d Sidki.
\newblock On the {B}urnside problem for periodic groups.
\newblock {\em Math. Z.}, 182(3):385--388, 1983.

\bibitem{KambitesSilvaSteinberg06}
Mark Kambites, Pedro~V. Silva, and Benjamin Steinberg.
\newblock The spectra of lamplighter groups and {C}ayley machines.
\newblock {\em Geom. Dedicata}, 120:193--227, 2006.

\bibitem{Klimann18}
Ines Klimann.
\newblock To infinity and beyond.
\newblock In {\em 45th {I}nternational {C}olloquium on {A}utomata, {L}anguages,
  and {P}rogramming}, volume 107 of {\em LIPIcs. Leibniz Int. Proc. Inform.},
  pages Art. No. 131, 12. Schloss Dagstuhl. Leibniz-Zent. Inform., Wadern,
  2018.

\bibitem{KrohnRhodes68}
Kenneth Krohn and John Rhodes.
\newblock Complexity of finite semigroups.
\newblock {\em Ann. of Math. (2)}, 88:128--160, 1968.

\bibitem{MacedonskaNekrashevychSushchansky00}
O.~Macedo\'{n}ska, V.~Nekrashevych, and V.~Sushchansky.
\newblock Commensurators of groups and reversible automata.
\newblock {\em Dopov. Nats. Akad. Nauk Ukr. Mat. Prirodozn. Tekh. Nauki},
  (12):36--39, 2000.

\bibitem{SilvaSteinberg05}
P.~V. Silva and B.~Steinberg.
\newblock On a class of automata groups generalizing lamplighter groups.
\newblock {\em Internat. J. Algebra Comput.}, 15(5-6):1213--1234, 2005.

\bibitem{SkipperSteinberg20}
Rachel Skipper and Benjamin Steinberg.
\newblock Lamplighter groups, bireversible automata, and rational series over
  finite rings.
\newblock {\em Groups Geom. Dyn.}, 14(2):567--589, 2020.

\bibitem{SteinbergVorobetsVorobets11}
Benjamin Steinberg, Mariya Vorobets, and Yaroslav Vorobets.
\newblock Automata over a binary alphabet generating free groups of even rank.
\newblock {\em Internat. J. Algebra Comput.}, 21(1-2):329--354, 2011.

\bibitem{VorobetsVorobets10}
Mariya Vorobets and Yaroslav Vorobets.
\newblock On a series of finite automata defining free transformation groups.
\newblock {\em Groups Geom. Dyn.}, 4(2):377--405, 2010.

\end{thebibliography}

\end{document}